\documentclass[12pt,reqno]{amsart}

\usepackage{graphicx}
     \makeatletter
     \def\section{\@startsection{section}{1}%
      \z@{.7\linespacing\@plus\linespacing}{.5\linespacing}%
     {\bfseries
     \centering
     }}
     \def\@secnumfont{\bfseries}
     \makeatother
\newtheorem{theorem}{Theorem}[section]

\newtheorem{proposition}{Proposition}[section]

\theoremstyle{definition}

\theoremstyle{remark}

\def \lim   {\text {\rm lim}}

\begin{document}
\title[]{On some positive definite functions} 

\author[Rajendra Bhatia]{Rajendra Bhatia}
\address{Indian Statistical Institute, New Delhi-110016, India}
\address{Sungkyunkwan University, Suwon 440-746, Korea}
\email{rbh@isid.ac.in}

\author[Tanvi Jain]{Tanvi Jain}

\address{Indian Statistical Institute, New Delhi-110016, India}
\email{tanvi@isid.ac.in}

\begin{abstract}
We study the function $(1 - \|x\|)\slash (1 - \|x\|^r),$  and its 
reciprocal, on the Euclidean space $\mathbb{R}^n,$ with respect to properties 
like being positive definite, conditionally positive definite, and infinitely 
divisible.
\end{abstract}

\maketitle

\section{Introduction} 
For each $n \ge 1,$ consider the space $\mathbb{R}^n$ with the Euclidean norm 
$\|\cdot\|.$ According to a classical theorem going back to Schoenberg \cite{ijs2} and much used in interpolation theory (see, e.g., \cite{m}), the 
function $\varphi (x) = \|x\|^r$ on $\mathbb{R}^n,$ for any $n,$ is 
conditionally negative definite if and only if $0 \leq r \leq 2.$ It follows 
that if $r_j,$ $1 \leq j \leq m,$ are real numbers with $0 \leq r_j \leq 2,$ 
then the function
\begin{equation}
 g (x) = 1 + \| x \|^{r_{1}} + \cdots + \|x\|^{r_{m}}   \label{eq1}
\end{equation}
is conditionally negative definite, and by another theorem of Schoenberg, 
(see the statement {\bf S5} in Section 2 below), the function
\begin{equation}
f(x) = \frac{1}{1 + \| x \|^{r_{1}} + \cdots + \|x\|^{r_{m}}}    \label{eq2}
\end{equation}
is infinitely divisible. (A nonnegative function $f$ is called infinitely 
divisible if for each $\alpha > 0$ the function $f(x)^{\alpha}$ is positive 
definite.) We also know that for any $r > 2,$ the function $\varphi (x)= 
1\slash(1+\|x\|^r)$ cannot be positive definite. (See, e.g., Corollary 5.5.6 
of 
\cite{rbh}.)

With this motivation we consider the function

\begin{equation}
f (x) = \frac{1}{1+ \|x\| + \|x\|^2 + \cdots + \|x\|^m}, \,\,\, m \ge 1,    
\label{eq3}
\end{equation}
and its reciprocal, and study their properties related to positivity. More 
generally, we study the function

\begin{equation}
 f(x) = \frac{1-\|x\|}{1 - \|x\|^r}, \,\,\, r > 0,   \label{eq4}
\end{equation}
and its reciprocal. As usual, when $\|x\|=1$ the right-hand side of \eqref{eq4} 
is interpreted as the limiting  value $1/r.$ This convention will be followed 
throughout the paper.  The function \eqref{eq3} is the special case of 
\eqref{eq4} 
when $r = m + 1.$ 

Our main results are the following.
\begin{theorem}\label{thm1}
  Let $0<r\leq 1.$  Then for each $n,$  the 
function $f(x) = \frac{1 - 
\|x\|}{1-\|x\|^r}$ on $\mathbb{R}^n$ is conditionally negative definite. As a 
consequence, the function $g(x) = \frac{1-\|x\|^r}{1 - \|x\|}$ is infinitely 
divisible.
\end{theorem}
The case $r \ge 1$ turns out to be more intricate.
\begin{theorem}\label{thm2}
Let $n$ be any natural number. Then the function $g(x) = \frac{1 - 
\|x\|^r}{1-\|x\|}$  on $\mathbb{R}^n$ is conditionally negative definite if and 
only if $1 \leq r \leq 3.$ As a consequence the function $f(x) = 
\frac{1 - \|x\|}{1-\|x\|^r}$ is infinitely divisible for $1 \leq r \leq 3.$
\end{theorem}

In the second part of Theorem \ref{thm2} the condition $1 \leq r \leq 3$ is 
sufficient but not necessary. We will show that the function $f$ is infinitely 
divisible for $1 \leq r \leq 4.$ On the other hand we show that when $r=9,$ $f$ 
need not even be positive definite for all $n.$

In the case $n=1$ we can prove the following theorem.

\begin{theorem}\label{thm3}
For every $1 \leq r < \infty$ the function $f(x) = \frac{1 - |x|}{1-|x|^r}$ on 
$\mathbb{R}$ is positive definite.
\end{theorem}

\section{Some classes of matrices and functions}
 Let $A=[a_{ij}]$ be an $n \times n$ real symmetric matrix. Then $A$ is said to 
be {\it positive semidefinite (psd)} if  $\langle x, Ax \rangle \ge 0$ for all 
$x \in \mathbb{R}^n,$ {\it conditionally positive definite (cpd)} if $\langle 
x, Ax \rangle \ge 0$ for all $x \in \mathbb{R}^n$ for which $\sum x_j=0,$ and 
{\it conditionally negative definite (cnd)} if $-A$ is cpd. If $a_{ij} \ge 0,$ 
then for any real number $r,$ we denote by $A^{\circ r}$ the $r$th {\it 
Hadamard power} of $A;$ i.e., $A^{\circ r} = [a_{ij}^r].$ If $A^{\circ r}$ is 
psd for all  $r \ge 0,$ we say that $A$ is {\it infinitely divisible.} 

Let $f : \mathbb{R} \rightarrow \mathbb{R}$ be a continuous function. We say 
$f$ is {\it positive definite} if for every $n,$ and for every choice of real 
numbers $x_1, x_2, \ldots, x_n,$ the $n \times n$ matrix $[f(x_i - x_j)]$ is 
psd. In the same way, $f$ is called cpd, cnd, or infinitely divisible if the 
matrices $[f(x_i - x_j)]$ have the corresponding property.

Next, let $f$ be a nonnegative $C^{\infty}$ function on the positive half line 
$(0, \infty).$ Then $f$ is called {\it completely monotone} if 
\begin{equation}
(-1)^n f^{(n)} (x) \ge 0 \quad \mbox{for all} \quad n \,\ge\, 0. \label{eq5}
\end{equation}
According to a theorem of Bernstein and Widder, $f$ is completely monotone if and only if it can be represented as
\begin{equation*}
f(x)=\int_0^\infty e^{-tx}\, d\mu (t),
\end{equation*}
where $\mu$ is a positive measure. 
$f$ is called a {\it Bernstein function} if its derivative $f^{\prime}$ is 
completely monotone; i.e., if

\begin{equation}
(-1)^{n-1} f^{(n)} (x) \ge 0 \quad \mbox{for all} \quad n \ge 1.    \label{eq6}
\end{equation}
Every such function can be expressed as

\begin{equation}
f(x) = a + b x + \int_0^{\infty} (1-e^{-tx}) d\mu (t),    \label{eq7}
\end{equation}
where $a, b \ge 0$ and $\mu$ is a measure satisfying the condition 
$\int_0^{\infty} (1 \wedge t) \,\, d\mu (t) < \infty.$ If this measure $\mu$ is 
absolutely continuous with respect to the Lebesgue  measure, and the associated 
density $m(t)$ is a completely monotone function,  then we say that $f$ is a 
{\it complete Bernstein function}.

The class of complete Bernstein functions coincides with the class of {\it Pick 
functions} (or {\it operator monotone functions}). Such a function has an 
analytic continuation to the upper half-plane $\mathbb{H}$  with the property 
that $\mbox{Im}\,\,f(z) \ge 0$ for all  $z \in \mathbb{H}.$ See Theorem 6.2 in 
\cite{rsrszv}.

For convenience we record here some basic facts used in our proofs.These can be 
found in the comprehensive monograph \cite{rsrszv}, or in the survey paper \cite{cb}.
\begin{itemize}
 \item[{\bf S1.}] A function $\varphi$ on $(0, \infty)$ is completely 
monotone, if and only if the function $f(x) = \varphi (\|x\|^2)$ is continuous 
and positive definite on $\mathbb{R}^n$ for every $n \ge 1.$
 
 \item[{\bf S2.}] A function $\varphi$ on $(0, \infty)$ is a Bernstein function 
if and only if the function $f(x) = \varphi (\|x\|^2)$ is continuous and cnd on 
$\mathbb{R}^n$ for every $n \ge 1.$

 \item[{\bf S3.}] If $f$ is a Bernstein function, then $1/f$ is completely 
monotone.

 \item[{\bf S4.}] If $f$ is a Bernstein function, then for each $0 < \alpha < 
1,$ the functions $f(x)^{\alpha}$ and $f(x^{\alpha})$ are also Bernstein. If 
$f$ is completely monotone, then $f(x^{\alpha})$ has the same property for $0 < 
\alpha < 1.$
\item[{\bf S5.}] A function $f$ on $\mathbb{R}$ is cnd if and only if $e^{-tf}$ is positive definite for every $t>0$. Combining this with the Bernstein-Widder theorem, we see that if $f$ is a nonnegative cnd function and $\varphi$ is completely monotone, then the composite function $\varphi\circ f$ is positive definite. In particular, if $r>0$, and we choose $\varphi(x)=x^{-r},$ we see that the function $f(x)^{-r}$ is positive definite. In other words $1/f$ is infinitely divisible.  
\end{itemize}

\section{Proofs and Remarks}
Our proof of Theorems \ref{thm1} and \ref{thm2} relies on the following 
proposition. This is an extension of results of T. Furuta \cite{tf} and 
F. Hansen \cite{fh}. 

\begin{proposition}\label{prop1}
Let $p,q$ be positive numbers with $0 < p \leq 1,$ and $p \leq q \leq p+1.$ 
Then the function $f(x) = (1 - x^q)/(1 - x^p)$ on the positive half-line is 
operator monotone. 
\end{proposition}

\begin{proof}
The case $p=q$ is trivial; so assume $p<q.$ It is convenient to use the 
formula
\begin{equation}
\frac{1 - x^q}{1-x^p} =\frac{q}{p} \,\,\int_0^1 \,\,\left ( \lambda 
\,\,x^p 
+ 1 - \lambda \right )^{\frac{q-p}{p}} \,\, d \lambda,
    \label{eq8}
\end{equation}
which can be easily verified. If $z$ is a complex number with $\mbox{Im}\,\,z > 
0,$ then for $0 < \lambda < 1,$ the number $\lambda z^p + 1 - \lambda$ lies in 
the 
sector $\left \{w : 0 < \mbox{Arg} \,\,w < p \pi  \right \}.$  Since $0 < 
\frac{q-p}{p} \leq \frac{1}{p},$ we see that $\left (\lambda z^p + 1 - 
\lambda \right )^{\frac{q-p}{p}}$ lies in the upper half-plane. This shows that 
the function represented by \eqref{eq8} is a Pick function.
\end{proof}

Now let $0 < r \leq 1.$ Choosing $p = r/2$ and $q = 1/2,$ we see from 
Proposition \ref{prop1} that the function $\varphi (x) = 
\frac{1-x^{1/2}}{1-x^{r/2}}$ is operator monotone. Appealing to fact {\bf S2} 
we obtain Theorem \ref{thm1}.

Next let $1 \leq r \leq 3.$ Choosing $p=1/2$ and $q=r/2,$ we see from 
Proposition \ref{prop1} that the function $\varphi (x) = 
\frac{1-x^{r/2}}{1-x^{1/2}}$ is operator monotone. Again appealing to {\bf S2} 
we see that the function $g(x) = \frac{1 - \|x\|^r}{1 - \|x\|}$ is cnd on the 
Euclidean space $\mathbb{R}^n$ for every $n.$

The necessity of the condition $1 \leq r \leq 3$ is brought out by the 
L\'evy-Khinchine formula. A continuous function $g:\mathbb{R} \rightarrow 
\mathbb{C}$ is cnd if and only it can be represented as
$$g(x) = a + ibx + c^2 x^2 + \int_{\mathbb{R}\backslash \{0\}} \left ( 1 - 
e^{itx} + \frac{itx}{1+t^2} \right )  d \nu (t), $$
where $a,b,c$ are  real numbers, and $\nu$ is a positive measure on 
$\mathbb{R}\backslash \{0\}$ such that $\int (t^2 \slash (1+t^2)) d\nu(t) < 
\infty.$ See \cite{rsrszv}. It is clear then that $g(x) = O (x^2)$ at $\infty.$ So, if 
$r > 3,$ the function $g(x)$ of Theorem \ref{thm2} cannot be cnd on 
$\mathbb{R}.$ This proves Theorem \ref{thm2} completely.

Now we show that $f(x) = \frac{1 - \|x\|}{1 - \|x\|^r}$ is infinitely divisible 
for $1 \leq r \leq 4.$ The special case $r=4$ is easy. We have
$$\frac{1-\|x\|}{1-\|x\|^4} = 
\frac{1}{1+\|x\|+\|x\|^2+\|x\|^3}=\frac{1}{1+\|x\|} \,\, \frac{1}{1+\|x\|^2}, $$
and we know that both $\frac{1}{1+\|x\|}$ and $\frac{1}{1+\|x\|^2}$ are 
infinitely divisible, and therefore so is their product. The general case is 
handled as follows.

 By Proposition \ref{prop1}, the function $\frac{1-x^r}{1-x}$ is operator 
monotone for $1 \leq r \leq 2.$ Repeating our arguments above, we see that 
$\frac{1-\|x\|^2}{1-\|x\|^{2r}}$ is an infinitely divisible function for $1\leq 
r \leq 2.$ We know that $\frac{1}{1+\|x\|}$ is infinitely divisible; hence so 
is the product
$$\frac{1-\|x\|^2}{1-\|x\|^{2r}} \,\,\frac{1}{1+\|x\|} 
= \frac{1-\|x\|}{1-\|x\|^{2r}}, \quad 1 \leq r \leq 2. $$
In other words $\frac{1-\|x\|}{1-\|x\|^{r}}$ is infinitely divisible for 
$2 \leq r \leq 4.$

We now consider what happens for $r > 4.$ In the special case $n=1,$ Theorem 
\ref{thm3} says that this function is at least positive definite for all 
$r>4.$ By a theorem 
of P\'olya (see \cite{rbh}, p.151) any continuous, nonnegative, even function on 
$\mathbb{R}$ which is convex and monotonically decreasing on $[0, \infty)$ is 
positive definite. So Theorem \ref{thm3} follows from the following proposition.

\begin{proposition}\label{prop2}
The function
 \begin{equation}
f(x) = \frac{1-x}{1-x^r}, \quad 1 < r < \infty,    \label{eq9}
\end{equation}
on the positive half-line $(0, \infty)$ is monotonically decreasing and convex.
\end{proposition}

\begin{proof}
A calculation shows that
\begin{equation}
f^{\prime}(x) = \frac{(1-r) x^r + r x ^{r-1} -1}{(1-x^r)^2},    \label{eq10}
\end{equation}
and
\begin{eqnarray}
f^{\prime \prime}(x) &=& \frac{1}{(1-x^r)^3} \left \{ r (1-r) x^{2r-1} + r 
(1+r)x^{2r-2} \right 
. \nonumber \\
&& \left . - r(1+r) x^{r-1} - r (1-r) x^{r-2}\right \}. \nonumber \\
&=& \frac{1}{(1-x^r)^3} \,\,\, \varphi(x), \quad \mbox{say.} \label{eq11} 
\end{eqnarray}
Since $f^{\prime \prime}(x)$ is well-defined at $1,$ the function $\varphi$ 
must have a zero of order at least three at $1.$ On the other hand, 
by the Descartes rule of signs, (see \cite{gpgs},p.46), $\varphi(x)$ can have at most three positive zeros.
Thus the only zero of $\varphi$ in $(0, \infty)$ is at the point 
$x=1.$

Next note that when $x$ is small, the last term of $\varphi(x)$ is dominant, 
and therefore $\varphi(x) > 0.$ On the other hand, when $x$ is large, the first 
term of $\varphi (x)$ is dominant, and therefore $\varphi(x) < 0.$ Thus 
$\varphi(x)$ is positive if $x < 1,$ and negative if $x > 1.$ This shows that 
$f^{\prime \prime} (x) \ge 0.$ Hence $f$ is convex. Since $f(0) = 1,$ and 
$\underset{x \rightarrow \infty}{\lim} f(x) = 0,$ this  also shows that $f$ is 
monotonically decreasing, a fact which can be easily seen otherwise too. 
\end{proof}

Does the function $f$ in \eqref{eq9} have any stronger convexity properties? We 
have seen that if $1 \leq r \leq 2,$ then the reciprocal of $f$ is operator 
monotone. Hence by fact {\bf S3}, $f$ is completely monotone for $1 \leq r \leq 
2.$ For $r > 2,$ however $f$ is not even $\log$-convex.

Recall that a nonnegative function $f$ on $(0, \infty)$ is called 
$\log$-{\it convex} if $\log \,f$ is convex. If $f^{\prime},$ $f^{\prime 
\prime}$ exist, this condition is equivalent to
\begin{equation}
(f^{\prime} (x))^2 \leq f (x) \,\, f^{\prime \prime}(x) \quad \mbox{for 
all}\quad x.    
\label{eq12}
\end{equation}
(See \cite{fwskvh},p.485). A completely monotone function is $\log$-convex.

\begin{proposition}\label{prop3}
The function $f(x) = \frac{1-x}{1-x^r}$ on $(0, \infty)$ is $\log$-convex if 
and only if $1 \leq r \leq 2.$ 
\end{proposition}

\begin{proof}
From the expressions \eqref{eq9}, \eqref{eq10} and \eqref{eq11} we see that
\begin{equation}
 f(x) f^{\prime \prime} (x) - (f^{\prime}(x))^2 = \frac{\psi (x)}{(1-x^r)^4},   
\label{eq13}
\end{equation}
where
\begin{eqnarray}
\psi(x) &=& (r-1) x^{2r} - 2r x^{2 r-1} + rx^{2r-2} + (r^2 - r + 2) x^r 
\nonumber \\
&& - 2r(r-1) x^{r-1} - 1 + r (r-1) x^{r-2}.    \label{eq14}
\end{eqnarray}
Using condition \eqref{eq12} we see from \eqref{eq13} that $f$ is $\log$-convex 
if and only if $\psi (x) \ge 0$ for all $x.$ If $r>2,$ it is clear from 
\eqref{eq14} that $\psi (0)=-1,$ and $\psi$ is negative in a neighbourhood of 
$0.$ So $f$ is not $\log$-convex.

We have already proved that when $1 < r < 2,$ $f$ is completely monotone, and 
hence $\log$-convex. It is instructive to see how the latter property can be 
derived easily using the condition \eqref{eq12}. It is clear from \eqref{eq13} 
that $\psi$ must have a zero of order at least $4$ at $1.$ On the other hand, 
there are just four sign changes in the coefficients on the right-hand side of 
\eqref{eq14}. So by the Descartes rule of signs (\cite{gpgs},p.46) $\psi$ 
has at most four positive zeros. Thus $\psi$ has only one zero, it is at $1$ 
and has multiplicity four. The coefficients of both $x^{2r}$ and $x^{r-2}$ in 
\eqref{eq14} are positive. Hence $\psi$ is always nonnegative.
\end{proof}
\vskip0.1in
Because of {\bf S1}, the function $f(x) = \frac{1-\|x\|}{1-\|x\|^r}$ would be 
positive definite on $\mathbb{R}^n$ for every $n,$ if and only if the function
\begin{equation}
h(x) = \frac{1-x^{1/2}}{1-x^{r/2}},  \label{eq15}
\end{equation}
on $(0, \infty)$ were completely monotone. From {\bf S4} we see that this would 
be a consequence of the complete monotonicity of the function $f(x) = 
\frac{1-x}{1-x^r};$ but the latter holds if and only if $1 \leq r \leq 2.$ We 
now show that when $r=9,$ the function $h$ in \eqref{eq15} is not even $\log$ 
convex.

For this we use the fact that $h$ is $\log$ convex if and only if
\begin{equation}
h \left ( \frac{x+y}{2}\right )^2 \leq h (x) h(y) \quad \mbox{for 
all}\,\,\,x,y.   \label{eq16}
\end{equation}
Choose $x=9/25,$ $y=16/25.$ Then $\frac{x+y}{2} = 1/2.$ When $r=9,$ the 
function $h$ in \eqref{eq15} reduces to 
$$h(x) = \left ( \sum_{j=0}^{8} \,\, x^{j/2} \right )^{-1}.$$
So, the inequality \eqref{eq16} would be true for the chosen values of $x$ and 
$y,$ if we have
$$\sum_{j=0}^{8} \left (\frac{3}{5} \right )^j  \sum_{j=0}^{8} \left 
(\frac{4}{5} \right )^j  \leq \left (\sum_{j=0}^{8} \left (\frac{1}{\sqrt{2}} 
\right )^j \right )^2. $$
A calculation shows that this is not true as, up to the first decimal 
place, the left-hand side is 10.7 and the right-hand side is 10.6.

We are left with some natural questions:

\begin{enumerate}
 \item[1.] What is the smallest $r_0$ for which the function $f$ of Theorem 
\ref{thm2} is not infinitely divisible (or positive definite) for all 
$\mathbb{R}^n?$ Our analysis shows that $4 < r_0 < 9.$

 \item[2.] What is the smallest $n_0$ for which there exists some $r>4,$ such 
that this function $f$ is not positive definite on $\mathbb{R}^{n_{0}}?$

 \item[3.] Is the function $f$ in Theorem \ref{thm3} infinitely divisible on 
$\mathbb{R}?$ By Theorem 10.4 in \cite{fwskvh} a sufficient condition for this to 
be true is $\log$ convexity of the function $\frac{1-x}{1-x^r}$ on $(0, 
\infty).$ We have seen that this latter condition holds if and only if $1 \leq 
r \leq 2.$ Note that we have shown by other arguments that $f$ is infinitely 
divisible for $1 \leq r \leq 4.$
\end{enumerate}

\noindent Several examples of infinitely divisible functions arising in 
probability theory are listed in \cite{fwskvh}. Many more with origins in our study 
of operator inequalities can be found in \cite{rbhhk} and \cite{hk}. It was observed already in \cite{bs} that the function defined in \eqref{eq2} is infinitely divisible. 

\vskip0.2in

{\it The work of the first author is supported by a J. C. Bose National Fellowship, and of the second author by an SERB Women Excellence Award. The first author was a Fellow Professor at Sungkyunkwan University in the summer of 2014.}

\end{document}